\documentclass[10pt]{amsart}
\usepackage{amsmath}
\usepackage{amsfonts}
\usepackage{amssymb}
\usepackage{graphicx}
\usepackage{bm}

\newcommand\eps{\varepsilon}
\newcommand\R{{\mathbf{R}}}
\newcommand\C{{\mathbf{C}}}
\newcommand\Z{{\mathbf{Z}}}

\theoremstyle{plain}
  \newtheorem{theorem}[subsection]{Theorem}

  \newtheorem{lemma}[subsection]{Lemma}
  \newtheorem{corollary}[subsection]{Corollary}

\theoremstyle{remark}
  \newtheorem{remark}[subsection]{Remark}

\theoremstyle{definition}
  \newtheorem{definition}[subsection]{Definition}

\begin{document}

\title[]{Maximizers for the Strichartz inequalities and the Sobolev-Strichartz inequalities for the Schr\"odinger equation}
\author{Shuanglin Shao}
\address{Department of Mathematics, UCLA, CA 90095}
\email{slshao@math.ucla.edu}


\vspace{-0.1in}
\begin{abstract}
In this paper, we first show that there exists a maximizer for the
non-endpoint Strichartz inequalities for the Schr\"odinger
equation in all dimensions based on the recent linear profile
decomposition results. We then present a new proof of the linear
profile decomposition for the Schr\"oindger equation with initial
data in the homogeneous Sobolev space; as a consequence, there
exists a maximizer for the Sobolev-Strichartz inequality.
\end{abstract}
\maketitle

\section{Introduction}
We consider the free Schr\"odinger equation
\begin{equation}\label{eq:free-schrodinger}
i\partial_t u+\Delta u=0,
\end{equation}
with initial data $u(0,x)=u_0(x)$ where $u:\R\times \R^d\to \C$ is
a complex-valued function and $d\ge 1$. We can denote the solution $u$ by
using the Schr\"odinger evolution operator $e^{it\Delta}$:
\begin{equation}\label{eq:schr-evolution}
u(t,x):=e^{it\Delta}u_0(x):= \int_{\R^d}
e^{ix\cdot\xi-it|\xi|^2}\hat{u}_0(\xi)d\xi,
\end{equation} where $\hat{u}_0$ is the spatial Fourier transform defined via
\begin{equation}\label{eq:def-fourier-transfm}
\hat{u}_0(\xi):=\int_{\R^d}e^{-ix\cdot\xi}u_0(x)dx,
\end{equation} where $x\cdot\xi$ (abbr. $x\xi$) denotes the Euclidean inner product of $x$ and $\xi$ in the spatial space $\R^d$. Formally the solutions to this equation have a conserved mass
\begin{equation}\label{eq:mass-consv}
\int_{\R^d} |u(t,x)|^2 dx.
\end{equation}
A family of well-known inequalities, the Strichartz inequalities, is
associated with \eqref{eq:free-schrodinger}, which is very useful
in the linear and nonlinear dispersive equations. It asserts that, for any $u_0\in L^2_x(\R^d)$,  there exists a constant $C_{d,q,r}>0$ such that
\begin{equation}\label{eq:Strichartz-inequ.}
\|e^{it\Delta}u_0\|_{L^q_tL^r_x(\R\times \R^d)}\le C_{d,q,r}\|u_0\|_{L^2_x(\R^d)}
\end{equation}
holds if and only if $q$ and $r$ are Schr\"odinger admissible, i.e., \begin{equation}\label{eq:schrod-admiss}
\frac 2q+\frac dr=\frac d2, \quad (q,r,d)\neq(2,\infty, 2),\quad q, r\ge2.
\end{equation}
 The pairs $(q,r)=(2, \frac {2d}{d-2})$ when $d\ge 3$ or $(q,r)=(4,\infty)$ when $d=1$, \eqref{eq:Strichartz-inequ.} is referred to as  the `` endpoint" estimate, otherwise the ``non-endpoint" estimate for the rest pairs. It has  a long history to establish \eqref{eq:Strichartz-inequ.} for all Schr\"odinger admissible pairs in \eqref{eq:schrod-admiss} expect when $(q,r)=(\infty,2)$, in which case it follows from \eqref{eq:mass-consv}. For the symmetric exponent $q=r=2+\frac 4d$, Strichartz established this inequality in \cite{Strichartz:1977} which in turn had precursors in \cite{Tomas:1975:restrict}. The non-endpoints were established by Ginibre and Velo \cite{Ginibre-Velo:1992:non-endpoint-Strichartz-inequality}, see also \cite[Theorem 2.3]{Tao:2006-CBMS-book} for a proof; the delicate endpoints were treated by Keel and Tao \cite{Keel-Tao:1998:endpoint-strichartz}. When $(q,r,d)=(2,\infty, 2)$, it has been known to fail, see e.g., \cite{Montgomery:1998:schrod} and \cite{Tao:2006:counterexample-bilinear-strichartz}.

A close relative of the Strichartz inequality for the Schr\"odinger equation is the
Sobolev-Strichartz inequality: for any $2\le q<\infty$, and $2\le r<\infty$ and $u_0\in \dot{H}^{s(q,r)}_x(\R^d)$ with $s(q,r):=\frac d2-\frac 2q-\frac dr>0$, there exists a constant $C_{d,q,r}>0$ such that
\begin{equation}\label{eq:sobolev-strichartz}
\|e^{it\Delta}u_0\|_{L^q_tL^r_x(\R\times \R^d)}\le C_{d,q,r} \|u_0\|_{\dot{H}_x^{s(q,r)}(\R^d)},
\end{equation} which can be proven by using the usual Sobolev embedding and the Strichartz inequality \eqref{eq:Strichartz-inequ.}.

In this paper, we are interested in the existence of maximizers for the Strichartz inequality \eqref{eq:Strichartz-inequ.} and the Sobolev-Strichartz inequality \eqref{eq:sobolev-strichartz}, i.e., functions which optimize \eqref{eq:Strichartz-inequ.} and \eqref{eq:sobolev-strichartz} in the sense that they become equal.

The answer to the former is confirmed for their non-endpoint estimates by an application of a recent powerful result, the profile decomposition for Schr\"odinger equations, which was developed in \cite{Bourgain:1998:refined-Strichartz-NLS}, \cite{Merle-Vega:1998:profile-schrod}, \cite{Carles-Keraani:2007:profile-schrod-1d}, \cite{Begout-Vargas:2007:profile-schrod-higher-d} and had many applications in nonlinear dispersive equations, see \cite{Killip-Visan:2008:clay-lecture-notes} and the reference within. The problem of the existence of maximizers and of determining them explicitly for the symmetric Strichartz inequality when $q=r=2+\frac 4d$ has been intensively studied. Kunze \cite{Kunze:2003:maxi-strichartz-1d} treated the $d=1$ case and showed that maximizers exist by an elaborate concentration-compactness method; when $d=1,2$, Foschi \cite{Foschi:2007:maxi-strichartz-2d} explicitly determined the best constants and showed that the only maximizers are Gaussians by using the sharp Cauchy-Schwarz inequality and the space-time Fourier transform; Hundertmark and Zharnitsky \cite{Hundertmark-Zharnitsky:2006:maximizers-Strichartz-low-dimensions} independently obtained this result by an interesting representation formula of the Strichartz inequalities; recently, Carneiro \cite{Carneiro:2008:sharp-strichartz-norm} proved a sharp Strichartz-type inequality by following the arguments in \cite{Hundertmark-Zharnitsky:2006:maximizers-Strichartz-low-dimensions} and found its maximizers, which derives the results in \cite{Hundertmark-Zharnitsky:2006:maximizers-Strichartz-low-dimensions} as a corollary when $d=1,2$; Very recently, Bennett, Bez, Carbery and Hundertmark \cite{Bennett-Bez-Carbery-Hundertmark:2008:heat-flow-of-strichartz-norm} offered a new proof to determine the best constants by using the method of heat-flow.

The answer to the latter is true as well. The proof follows almost
along similar lines as in the $L^2_x$ case if we have an
analogous profile decomposition for initial data in the
homogeneous Sobolev spaces. We offer a new proof for this fact,
which we have not seen in the literature.

\subsection{}In this subsection, we investigate the existence of maximizers for the non-endpoint Strichartz inequalities. To begin, we recall the profile decomposition result in
\cite{Begout-Vargas:2007:profile-schrod-higher-d} in the notation
of the symmetry group which preserves the mass and the Strichartz
inequalities.
\begin{definition}[Mass-preserving symmetry group]\label{def:symmetry} For any phase $\theta\in \R/2\pi\Z$, scaling parameter $h_0>0$,
frequency $\xi_0\in \R^d$, space and time  translation parameters $x_0, t_0\in \R^d$, we define the unitary transformation
$g_{\theta, h_0,\xi_0,x_0,t_0}: L^2_x(\R^d)\to L^2_x(\R^d)$ by the formula
\begin{equation}\label{eq:gp-elments}
[g_{\theta,
h_0,\xi_0,x_0,t_0}\phi](x)=e^{i\theta}e^{ix\cdot\xi_0}e^{-it_0\Delta}[\frac
{1}{h_0^{d/2}}\phi(\frac {\cdot-x_0}{h_0})](x).
\end{equation}
We let $G$ be the collection of such transformations; $G$ forms a group.
\end{definition}
\begin{definition}\label{def:orthogonal}
For $j\neq k$, two sequences $\Gamma^j_n=(h^j_{n},\xi^j_n,
x^j_n,t^j_n)_{n\ge 1}$ and
$\Gamma^k_n=(h^k_n,\xi^k_n,x_k^n,t^k_n)_{n\ge 1}$ in $(0,\infty)
\times\R^d\times\R^d\times\R$ are said to be orthogonal if one of
the followings holds:
\begin{itemize}
\item $\lim_{n\to\infty}\left(\frac {h^k_n}{h^j_{n}}+\frac {h^j_n}{h^k_n}+h_n^j|\xi_n^j-\xi_n^k|\right)=\infty,$
\item $\lim_{n\to\infty}\left(\frac{|t_n^j-t_n^k|}{(h^j_n)^2}+\left|\frac{x^j_n-x^k_n}
{h^j_n}+\frac{t_n^k(\xi^k_n-\xi^j_n)}{h^j_n}\right|\right)=\infty.$
\end{itemize}
\end{definition}
We rephrase the linear profile decomposition theorem in
\cite{Begout-Vargas:2007:profile-schrod-higher-d} by  using the notation in Definition \ref{def:symmetry}.
\begin{theorem}\label{thm:lin-profile}
Let $\{u_n\}_{n\ge 1}$ be a bounded sequence in $L^2_x$. Then up to passing to a subsequence of $(u_n)_{n\ge 1}$, there exists a sequence of functions $\phi^j\in L^2_x$ and  group elements $(g_n^j)_{n\ge 1,j\ge 1}=g_{0,h_n^j,\xi_n^j,x_n^j,t_n^j}\in G$ with orthogonal $(h_n^j,\xi_n^j,x_n^j,t_n^j)$ such that for any $N\ge 1$, there exists $e_n^N\in L^2_x$,
\begin{equation}\label{eq:profile-decom}
u_n=\sum_{j=1}^{N}g_n^j(\phi^j)+e_n^N,
\end{equation}
with the error term having the asymptotically vanishing Strichartz
norm
\begin{equation}\label{eq:strich-err}
\lim_{N\to \infty}\lim_{n\to \infty}\|e^{it\Delta}e_n^N\|_{L^{2+4/d}_{t,x}}=0,
\end{equation} and the following orthogonality properties: for any $N\ge 1$,
\begin{equation}\label{eq:pf-L2-ortho}
\lim_{n\to\infty}
\left(\|u_n\|^2_{L^2_x}-(\sum_{j=1}^N\|\phi^j\|^2_{L^2_x}+\|e_n^N\|^2_{L^2_x})\right)=0,
\end{equation} for $j\neq k$,
\begin{equation}\label{eq:pf-Stri-ortho}
\lim_{n\to \infty}\|e^{it\Delta}g_n^j(\phi^j)e^{it\Delta}g_n^k(\phi^k)\|_{L^{1+2/d}_{t,x}}=0,
\end{equation}
\begin{equation}\label{eq:pf-L2-wk-ortho-1}
\lim_{n\to \infty}\langle g_n^j(\phi^j),g_n^k(\phi^k)\rangle_{L^2_x}=0,
\end{equation}for any $1\le j\le N$,
\begin{equation}\label{eq:pf-L2-wk-ortho-2}
\lim_{n\to \infty}\langle g_n^j(\phi^j),e_n^N\rangle_{L^2_x}=0.
\end{equation}
\end{theorem}
The first main result in this paper concerns on the existence of maximizers for the symmetric Strichartz inequality $L^2_x\to L^{2+4/d}_{t,x}$.
\begin{theorem}\label{thm:Stri-max}
There exists a maximizing function $\phi\in L^2_x$ such that,
$$\|e^{it\Delta}\phi\|_{L^{2+\frac 4{d}}_{t,x}}=S\|\phi\|_{L^2_x}$$ with $S:=\sup\{{\|e^{it\Delta}u_0\|_{{L^{2+\frac {4}{d}}_{t,x}}}}:\|u_0\|_{L^2_x}=1\}$
being the sharp constant.
\end{theorem}
The proof of this theorem uses Theorem \ref{thm:lin-profile} and the following crucial inequality in
\cite{Begout-Vargas:2007:profile-schrod-higher-d}:
 for any $N\ge 1$,
\begin{equation}\label{eq:Stri-ortho}
\lim_{n\to\infty}\|\sum_{j=1}^Ne^{it\Delta}g_n^j(\phi^j)\|^{2+4/d}_{L^{2+4/d}_{t,x}}
\le
\sum_{j=1}^N\lim_{n\to\infty}\|e^{it\Delta}\phi^j\|^{2+4/d}_{L^{2+4/d}_{t,x}}.
\end{equation}
\begin{remark}
The inequality \eqref{eq:Stri-ortho} is a consequence of
\eqref{eq:pf-Stri-ortho} by an interpolation argument in \cite{Begout-Vargas:2007:profile-schrod-higher-d}, which we will generalize in the proof of Lemma \ref{le:ortho-other-Stri}. When $d=1,2$, one can actually show that \eqref{eq:Stri-ortho} is an equality by using the fact that $2+4/d$ is an even integer.
\end{remark}
The inequality \eqref{eq:Stri-ortho} suggests a way to obtaining
similar claims as in Theorem \ref{thm:Stri-max} for other
non-endpoint Strichartz inequalities if the following lemma were established.
\begin{lemma}\label{le:ortho-other-Stri}Let $q,r$ be non-endpoint Schr\"odinger admissible pairs and $N\ge 1$. If $q\ge r$,
\begin{equation}\label{eq:Stri-ortho-1}
\lim_{n\to\infty}\|\sum_{j=1}^Ne^{it\Delta}g_n^j(\phi^j)\|^r_{L_t^qL^r_x}
\le
\sum_{j=1}^N\lim_{n\to\infty}\|e^{it\Delta}\phi^j\|^r_{L_t^qL^r_x};
\end{equation}
if $q\le r$,
\begin{equation}\label{eq:Stri-ortho-2}
\lim_{n\to\infty}\|\sum_{j=1}^Ne^{it\Delta}g_n^j(\phi^j)\|^q_{L_t^qL^r_x}
\le
\sum_{j=1}^N\lim_{n\to\infty}\|e^{it\Delta}\phi^j\|^q_{L_t^qL^r_x}.
\end{equation}
\end{lemma}
Indeed, this is the case. Together with Theorem \ref{thm:lin-profile} again, this lemma yields the following corollary.
\begin{corollary}\label{coro:Stri-max}Let $q,r$ be non-endpoint Schr\"odinger admissible pairs.
There exists a maximizing function $\phi\in L^2_x$ such that,
$$\|e^{it\Delta}\phi\|_{L_t^qL^r_x}=S_{q,r}\|\phi\|_{L^2_x}$$ with $S_{q,r}:=\sup\{\|e^{it\Delta}u_0\|_{L_t^qL^r_x}:\|u_0\|_{L^2_x}=1\}$
being the sharp constant.
\end{corollary}
 The proof of this corollary is similar to that used in Theorem \ref{thm:Stri-max} and thus will be omitted. Instead, we will focus on proving Lemma \ref{le:ortho-other-Stri}.
\begin{remark}
When $(q,r)=(\infty, 2)$, from the conservation of mass \eqref{eq:mass-consv}, we see that every $L^2_x$-initial data is a maximizer for the Strichartz inequality.
\end{remark}

\subsection{} In this subsection we concern on the existence of maximizers for the Sobolev-Strichartz inequality \eqref{eq:sobolev-strichartz} for the Schr\"odinger equation.
\begin{theorem}\label{thm:sobolev-strichartz-max}
Let $q,r$ be defined as in \eqref{eq:sobolev-strichartz}. Then
there exists a maximizing function $\phi\in \dot{H}_x^{s(q,r)}$
for \eqref{eq:sobolev-strichartz} with $C_{d,q,r}$ being the sharp
constant
$S^{q,r}:=\sup\{\|e^{it\Delta}u_0\|_{L_t^qL^r_x}:\|u_0\|_{\dot{H}^{s(q,r)}_x}=1\}.$
\end{theorem}
As we can see, it suffices to establish a profile decomposition result for initial data in $\dot{H}^{s(q,r)}_x$.
\begin{theorem}\label{thm:pf-homog-sobolev}
Let $s(q,r)$ be defined as in \eqref{eq:sobolev-strichartz} and
$\{u_n\}_{n\ge 1}$ be a bounded sequence in $\dot{H}^{s(q,r)}_x$.
Then up to passing to a subsequence of $(u_n)_{n\ge 1}$, there
exists a sequence of functions $\phi^j\in \dot{H}^s_x$ and  a
sequence of parameters $(h_n^j,x_n^j,t_n^j)$ such that for any
$N\ge 1$, there exists $e_n^N\in \dot{H}^s_x$,
\begin{equation}\label{eq:ss-profile-decom}
u_n=\sum_{j=1}^{N}e^{-it_n^j\Delta}\left(\frac {1}{(h_n^j)^{d/2-s}}\phi^j(\frac {\cdot-x_n^j}{h_n^j})\right)+e_n^N,
\end{equation}
with the parameters $(h_n^j,x_n^j,t_n^j)$ satisfying the following constraint: for $j\neq k$,
\begin{equation}\label{eq:ss-param-constrnt}
\lim_{n\to\infty} \left(\frac {h_n^j}{h_n^k}+\frac {h_n^k}{h_n^j}
+\frac {|t_n^j-t_n^k|}{(h_n^j)^2}+\frac {|x_n^j-x_n^k|}{h_n^j}\right)=\infty,
\end{equation}
and the error term having the asymptotically vanishing
Sobolev-Strichartz norm
\begin{equation}\label{eq:ss-profile-err}
\lim_{N\to \infty}\lim_{n\to \infty}\|e^{it\Delta}e_n^N\|_{L_t^qL^r_x}=0,
\end{equation} and the following orthogonality property: for any $N\ge 1$,
\begin{equation}\label{eq:ss-pf-L2-ortho}
\lim_{n\to\infty}
\left(\|u_n\|^2_{\dot{H}^s_x}-(\sum_{j=1}^N\|\phi^j\|^2_{\dot{H}^s_x}+\|e_n^N\|^2_{\dot{H}^s_x})\right)=0.
\end{equation}
\end{theorem}
When $s=1$ and $d\ge 3$, Keraani
\cite{Keraani:2001:profile-schrod-H^1} established Theorem
\ref{thm:pf-homog-sobolev} for the Schr\"odinger equation based on
the following Besov-type improvement of the Sobolev embedding
\begin{equation}\label{eq:keraani-impv-sobolev}
\|f\|_{L^{2d/(d-2)}_x}\lesssim
\|Df\|_{L^2_x}^{1-2/d}\|Df\|^{2/d}_{\dot{B}^{0}_{2,\infty}},
\end{equation} where $\|\cdot\|_{\dot{B}^{0}_{2,\infty}}$ is the Besov norm defined
via
$$\|f\|_{\dot{B}^0_{2,\infty}}:=\sup_{k\in \Z}\|f_k\|_{L^2_x}$$ with $f_k$ denoting the $k$-th
Littlewood-Paley piece defined via the Fourier transform
$\hat{f}_k:=\hat{f}1_{2^k\le |\xi|\le 2^{k+1}}$ \footnote{For a
rigorous definition of the Littlewood-Paley decomposition (or
dyadic decomposition) in terms of smooth cut-off functions, see
\cite[p.241]{Stein:1993}.}, and $D^s$ the fractional differentiation
operator defined via the inverse Fourier transform,
$$D^s f(x):=\int_{\R^d}e^{ix\xi}|\xi|^s \hat{f}(\xi)d\xi.$$ He followed the arguments
in \cite{Bahouri-Gerard:1999:profile-wave}
where Bahouri and G\'erard established the profile-decomposition result in
the context of the wave equation with initial data in
$\dot{H}^1_x(\R^3)$. Recently under the same constraints on $s$ and
$d$, Killip and Visan \cite{Killip-Visan:2008:clay-lecture-notes}
obtained the same result by relying on their interesting improved
Sobolev embedding involving the critical $L^{2d/(d-2)}_x$-norm on the
right-hand side:
\begin{equation}\label{eq:killip-visan-impv-sobobev}
\|f\|_{L^{2d/(d-2)}_x}\lesssim \|Df\|_{L^2_x}^{1-2/d}\sup_{k\in \Z}\|f_k\|_{L^{2d/(d-2)}_x}^{2/d}.
\end{equation} Note that \eqref{eq:killip-visan-impv-sobobev} implies \eqref{eq:keraani-impv-sobolev} by the usual Sobolev embedding. By following their approaches, we will generalize both Keraani's and Killip-Visan's improved $\dot{H}^1_x$-Sobolev embeddings to those with $\dot{H}^s_x$ norms where $\frac 1r+\frac sd=\frac 12$ and $d\ge 1$ in the appendix of this paper. Consequently almost same approaches as in \cite{Keraani:2001:profile-schrod-H^1} or \cite{Killip-Visan:2008:clay-lecture-notes} would yield Theorem \ref{thm:pf-homog-sobolev} without difficulties but we choose not to do it in this paper for simplicity. However, we will offer a new proof of Theorem \ref{thm:pf-homog-sobolev} by taking advantage of the existing $L^2_x$ linear profile decomposition, Theorem \ref{thm:lin-profile}. The idea can be roughly explained as follows.

For $(u_n)_{n\ge 1}\in \dot{H}^s_x$, we regard $(D^su_n)_{n\ge 1}$
as an $L^2_x$ sequence and then apply Theorem
\ref{thm:lin-profile} to this new sequence. Then the main task is
to show how to eliminate the frequency parameter $\xi_n^j$ from
the decomposition. To do it, we have two cases according to the
limits of the sequence $(h_n^j\xi_n^j)_{n\ge 1}$ for each $j$:
when the limit of $h_n^j\xi_n^j$ is finite, we will
change the profiles $\phi^j$ so that we can reduce to $\xi_n^j=0$;
on the other hand, when it is infinite, we will group this term
into the error term since one can show that its Sobolev Strichartz
norm is asymptotically small.

We organize this paper as follows: in Section \ref{sec:notations}
we establish some notations; in Section
\ref{sec:Stri-max} we prove Theorems \ref{thm:Stri-max},
\ref{le:ortho-other-Stri}; in Section
\ref{sec:sobolev-strichartz-max} we prove Theorem
\ref{thm:sobolev-strichartz-max}; finally in Appendix, we include
the arguments for the general Keraani's and Killip-Visan's
improved Sobolev embeddings.

\textbf{Acknowledgments.} The author is grateful to his advisor
Terence Tao for many helpful discussions. The author also thanks Professor Changxing Miao
for his comments.

\section{Notation}\label{sec:notations}
We use $X\lesssim Y$, $Y\gtrsim X$, or $X=O(Y)$ to denote the estimate $|X|\le
C Y$ for some constant $0<C<\infty$, which might depend on $d$,$p$ and $q$ but not on the functions. If $X\lesssim Y$ and $Y\lesssim X$ we will write $X\sim Y$. If the constant $C$ depends on a special parameter, we shall denote it explicitly by subscripts.

Throughout the paper, the limit sign $\lim_{n\to\infty}$ should be
understood as $\limsup_{n\to\infty}$.

The homogeneous Sobolev space $\dot{H}^s_x(\R^d)$ for $s\ge 0$ can
be defined in terms of the fractional differentiation:
\begin{equation*}
\dot{H}^s_x(\R^d):=\{f:\|f\|_{\dot{H}^s_x(\R^d)}:=\|D^{s}f\|_{L^2_x(\R^d)}
=\||\xi|^{s}\hat{f}\|_{L^2_\xi(\R^d)}<\infty\}.
\end{equation*}
We define the space-time norm $L^q_tL^r_x$ of $f$ on $\R\times
\R^d$ by
\begin{equation*}\|f\|_{L^q_tL^r_x(\R\times\R^d)}:=\left(\int_{\R}\left(\int_{\R^d}
|f(t,x)|^{r}d\,x\right)^{q/r}d\,t\right)^{1/q}, \end{equation*}
with the usual modifications when $q$ or $r$ are equal to
infinity, or when the domain $\R\times\R^d$ is replaced by a small
region. When $q=r$, we abbreviate it by $L^{q}_{t, x}$. Unless
specified, all the space-time integration are taken over $\R\times
\R^d$, and the spatial integration over $\R^d$.

The inner product $\langle\cdot,\cdot\rangle_{L^2_x}$ in the Hilbert space $L^2_x(\R^d)$ is defined via $$\langle f,g\rangle_{L^2_x}:=\int_{\R^d}f(x)\bar{g}(x)dx,$$
where $\bar{g}$ denotes the usual complex conjugate of $g$ in the complex plane $\C$.

\section{Maximizers for the symmetric Strichartz inequalities}\label{sec:Stri-max}
\begin{proof}[Proof of Theorem \ref{thm:Stri-max}]
We choose a maximizing sequence $(u_n)_{n\ge 1}$ with $\|u_n\|_{l^2_x}=1$, and then, up to a subsequence, decompose it into linear profiles as in Theorem \ref{thm:lin-profile}. Then from the asymptotically vanishing Strichartz norm \eqref{eq:strich-err}, we obtain that, for any $\eps >0$, there exists $n_0$ so that for all $N\ge n_0$ and $n\ge n_0$,
\begin{equation*}
S-\eps\le \|\sum_{j=1}^{N} e^{it\Delta}g_n^j(\phi^j)\|_{L^{2+4/d}_{t,x}}.
\end{equation*}
Thus from \eqref{eq:Stri-ortho}, there exists $n_1\ge n_0$ such that when $n, N\ge n_1$,
\begin{equation*}
S^{2+4/d}-2\eps\le\sum_{j=1}^{N}\|e^{it\Delta}\phi^j\|^{2+4/d}_{L^{2+4/d}_{t,x}}.
\end{equation*}
Choosing $j_0\in [1,N]$ such that $e^{it\Delta} \phi^{j_0} $
has the largest $L^{2+4/d}_{t,x}$ norm among $1\le j\le N$, we see
that, by the usual Strichartz inequality,
\begin{align}\label{eq:local-1}
S^{2+4/d}-2\eps \le \|e^{it\Delta} \phi^{j_0}\|^{4/d}_{L^{2+4/d}_{t,x}}\sum_{j=1}^{N}
\|e^{it\Delta} \phi^j\|^2_{L^{2+4/d}_{t,x}}\le S^{2+4/d}\|\phi^{j_0}\|^{4/d}_{L^2_x}\le S^{2+4/d}
\end{align} since \eqref{eq:pf-L2-ortho} gives
\begin{equation}\label{eq:local-2}\sum_{j=1}^\infty \|\phi^j\|^2_{L^2_x}\le \lim_{n\to\infty}\|u_n\|^2_{L^2_x}=1.\end{equation}
This latter fact also gives that
$\lim_{j\to\infty}\|\phi^j\|_{L^2_x}=0$, which together with
\eqref{eq:local-1} shows that $j_0$ must terminate before some fixed
constant which does not depend on $\eps$. Hence in
\eqref{eq:local-1} we can take $\eps$ to zero to obtain
\begin{equation*}
\|\phi^{j_0}\|_{L^2_x}=1.
\end{equation*}
This further shows that $\phi^j=0$ for all but $j=j_0$ from
\eqref{eq:local-2}. Therefore $\phi^{j_0}$ is a maximizer. Thus
the proof of Theorem \eqref{thm:Stri-max} is complete.
\end{proof}
We will closely follow the approach in \cite[Lemma
5.5]{Begout-Vargas:2007:profile-schrod-higher-d} to prove Lemma
\ref{le:ortho-other-Stri}.
\begin{proof}[Proof of Lemma \ref{le:ortho-other-Stri}] We only handle \eqref{eq:Stri-ortho-1} since the proof of \eqref{eq:Stri-ortho-2} is similar. By interpolating with \eqref{eq:pf-Stri-ortho}, we see that
for $j\neq k$ and non-endpoint Schr\"odinger admissible pairs $q, r$,
\begin{equation}\label{eq:loc-3}
\lim_{n\to\infty}\|e^{it\Delta}g_n^j(\phi^j)e^{it\Delta}g_n^k(\phi^k)\|_{L^{q/2}_tL^{r/2}_x}=0.
\end{equation}
Now we expand the left hand side of \eqref{eq:Stri-ortho-1} out,
which is equal to
\begin{align*}
&\left(\int\left(\int\left|\sum_{j=1}^N
e^{it\Delta}g_n^j(\phi^j)\right|^rdx\right)^{q/r}dt\right)^{r/q}\\
&\le\left(\int\left(\int\sum_{j=1}^N\left|e^{it\Delta}g_n^j(\phi^j)\right|^2|\sum_{l=1}^{N}e^{it\Delta}
g_n^l(\phi^l)|^{r-2}\right.\right.\\
&\quad+\left.\left.\sum_{k\neq j} |e^{it\Delta}g_n^j(\phi^j)||e^{it\Delta}g_n^k(\phi^k)|
|\sum_{l=1}^{N}e^{it\Delta}g_n^l(\phi^l)|^{r-2}dx\right)^{q/r}dt\right)^{r/q}\\
&\le\sum_{j=1}^N\left(\int\left(\int|e^{it\Delta}g_n^j(\phi^j)|^2
|\sum_{l=1}^{N}e^{it\Delta}g_n^l(\phi^l)|^{r-2}dx\right)^{q/r}dt\right)^{r/q}\\
&\quad +\sum_{k\neq j}\left(\int\left(\int|e^{it\Delta}g_n^j(\phi^j)|
|e^{it\Delta}g_n^k(\phi^k)||\sum_{l=1}^{N}e^{it\Delta}g_n^l
(\phi^l)|^{r-2}dx\right)^{q/r}dt\right)^{r/q}\\
&=:A+B.
\end{align*}
For $B$, since $\frac rq=\frac 2q+\frac {r-2}{q},\quad 1=\frac 2r+\frac {r-2}{r},$
the H\"older inequality yields
\begin{equation*} B\le \sum_{k\neq j}\|e^{it\Delta}g_n^j(\phi^j)e^{it\Delta}g_n^k
(\phi^k)\|_{L_t^{q/2}L_x^{r/2}}\|\sum_{l=1}^N e^{it\Delta}g_n^l(\phi^l)\|^{r-2}_{L_t^{q}L_x^{r}},
\end{equation*}
which goes to zero by \eqref{eq:loc-3} as $n$ goes to infinity.
Hence we are left with estimating $A$.

For $A$, since $q,r$ are in the non-endpoint region and $q\ge r$,
we have $2< r\le 2+4/d$, i.e., $0<r-2\le 4/d$. We let $s:=[r-2]$,
the largest integer which is less than $r-2$. Then because $0\le
r-2-s<1$,
\begin{align}\label{eq:loc-4}
&|e^{it\Delta}g_n^j(\phi^j)|^2|\sum_{l=1}^{N}e^{it\Delta}g_n^l(\phi^l)|^{r-2}\nonumber\\
&\quad\le
\sum_{l=1}^N|e^{it\Delta}g_n^j(\phi^j)|^2|\sum_{k=1}^{N}e^{it\Delta}g_n^k(\phi^k)|^s
|e^{it\Delta}g_n^l(\phi^l)|^{r-2-s}.
\end{align}
We now eliminate some terms in \eqref{eq:loc-4}. The first case we
consider is $l\neq j$: since $r-2-s<2$, we write
\begin{align*}
&|e^{it\Delta}g_n^j(\phi^j)|^2|\sum_{k=1}^{N}e^{it\Delta}g_n^k(\phi^k)|^s
|e^{it\Delta}g_n^l(\phi^l)|^{r-2-s}\\
&\quad =
|e^{it\Delta}g_n^j(\phi^j)|^{4+s-r}|\sum_{k=1}^{N}e^{it\Delta}g_n^k(\phi^k)|^s
|e^{it\Delta}g_n^j(\phi^j)e^{it\Delta}g_n^l(\phi^l)|^{r-2-s}.
\end{align*}
Then the H\"older inequality and \eqref{eq:loc-3} show that the
summation above goes to zero as $n$ goes to infinity. So we may
assume that $l=j$ and take out the summation in $l$ in
\eqref{eq:loc-4}. The second case we consider is when the terms in
the expansion of $|\sum_{k=1}^{N}e^{it\Delta}g_n^k(\phi^k)|^s$
contain two distinct terms:
\begin{align*}
|\sum_{k=1}^{N}e^{it\Delta}g_n^k(\phi^k)|^s&\le
\sum_{k=1}^{N}|e^{it\Delta}g_n^k(\phi^k)|^s+\\
&+ \sum_{k_1\neq k_2,\atop
k_1+\cdots+k_s=s}|e^{it\Delta}g_n^{k_1}(\phi^{k_1})e^{it\Delta}g_n^{k_2}(\phi^{k_2})|
\cdots|e^{it\Delta}g_n^{k_s}(\phi^{k_s})|.
\end{align*}
Again the interpolation argument and \eqref{eq:loc-3} show that
the second term above goes to zero as $n$ goes to infinity.
Combining these two cases, we reduce \eqref{eq:loc-4} to
\begin{equation*}
|e^{it\Delta}g_n^j(\phi^j)|^r+\sum_{k\neq j} |e^{it\Delta}g_n^j(\phi^j)|^{r-s}|e^{it\Delta}g_n^k(\phi^k)|^s.
\end{equation*}
For the second term above, we consider $r-s\le s$ and $r-s\ge s$;
it is not hard to see that it goes to zero as expected when $n$
goes to infinity. Therefore the proof of Lemma
\ref{le:ortho-other-Stri} is complete.
\end{proof}

\section{Maximizers for the Sobolev-Strichartz inequalities}\label{sec:sobolev-strichartz-max}
\begin{proof}[Proof of Theorem \ref{thm:pf-homog-sobolev}]
It is easy to see that $\{u_n\}\in \dot{H}^s_x(\R^d) $ implies
that $\{D^{s}u_n\}\in L^2_x(\R^d)$. We then apply Theorem
\ref{thm:lin-profile}: there exists a sequence of
$(\psi^j)_{j\ge 1}$ and orthogonal
$\Gamma_n^j=(h_n^j,\xi_n^j,x_n^j,t_n^j)$ so that, for any $N\ge 1$
\begin{equation}\label{eq:loc-6}
D^su_n=\sum_{j=1}^N g_n^j(\psi^j)+w_n^N,
\end{equation} with $w_n^N\in L^2_x$ and all the properties in Theorem \ref{thm:lin-profile} being satisfied. Without loss of generality,
we assume all $\psi^j$ to be Schwartz functions. We then rewrite it as
\begin{equation}\label{eq:loc-7}
u_n=\sum_{j=1}^N D^{-s}g_n^j(\psi^j)+D^{-s}w_n^N.
\end{equation}
Writing $e_n^N:=D^{-s}w_n^N$, we see that for $q,r$ in
\eqref{eq:sobolev-strichartz}, the Sobolev embedding and
\eqref{eq:strich-err} together give \eqref{eq:ss-profile-err}.
Next we show how to eliminate $\xi_n^j$ in the profiles.

{\bf Case 1.} Up to a subsequence, $\lim_{n\to\infty}h_n^j\xi_n^j=\xi^j\in \R^d$ for
some $1\le j\le N$. From the Galilean transform,
$$e^{it_0\Delta}(e^{i(\cdot)\xi_0}\phi)(x)=e^{ix\xi_0-it_0|\xi_0|^2}e^{it_0\Delta}\phi(x-2t\xi_0),$$
we see that
\begin{align*}
&e^{-it_n^j\Delta}\left(\frac
{1}{(h_n^j)^{d/2}}e^{ix_n^j\xi_n^j+it_n^j|\xi_n^j|^2}[e^{i(\cdot)h_n^j\xi_n^j}\psi^j](\frac{\cdot-x_n^j}{h_n^j})\right)(x-2t_n^j\xi_n^j)\\
&=e^{it_n^j|\xi_n^j|^2}e^{-it_n^j\Delta}\left(\frac
{1}{(h_n^j)^{d/2}}e^{i(\cdot)\xi_n^j}\psi^j(\frac{\cdot-x_n^j}{h_n^j})\right)(x-2t_n^j\xi_n^j)\\
&=e^{ix\xi_n^j}e^{-it_n^j\Delta}\left(\frac{1}{(h_n^j)^{d/2}}\psi^j(\frac{\cdot-x_n^j}{h_n^j})\right)(x)=g_n^j(\psi^j)(x).
\end{align*} On the other hand, since the
symmetries defined in Definition \ref{def:symmetry} keep the
$L^2_x$-norm invariant,
\begin{align*}
&\|e^{-it_n^j\Delta}\left(\frac
{1}{(h_n^j)^{d/2}}e^{ix_n^j\xi_n^j+it_n^j|\xi_n^j|^2}[e^{i(\cdot)h_n^j\xi_n^j}\psi^j](\frac{\cdot-x_n^j}{h_n^j})\right)(x-2t_n^j\xi_n^j)-\\
&\quad- e^{-it_n^j\Delta}\left(\frac
{1}{(h_n^j)^{d/2}}e^{ix_n^j\xi_n^j+it_n^j|\xi_n^j|^2}[e^{i(\cdot)\xi^j}\psi^j](\frac{\cdot-x_n^j}{h_n^j})\right)(x-2t_n^j\xi_n^j)\|_{L^2_x}\\
&=\|(e^{ixh_n^j\xi_n^j}-e^{ix\xi^j})\psi^j\|_{L^2_x}\to 0
\end{align*} as $n$ goes to infinity. Hence we can replace
$g_n^j(\psi^j)$ with $$e^{-it_n^j\Delta}\left(\frac
{1}{(h_n^j)^{d/2}}e^{i(\cdot)\xi_n^j+it_n^j|\xi_n^j|^2}[e^{i(\cdot)\xi^j}\psi^j](\frac{\cdot-x_n^j}{h_n^j})\right)(x-2t_n^j\xi_n^j);$$
for the differences, we put them into the error term. Thus if
further regarding $e^{ix\xi^j}\psi^j$ as a new $\psi^j$, we can
re-define
$$g_n^j(\psi^j):=e^{-it_n^j\Delta}\left(\frac
{1}{(h_n^j)^{d/2}}e^{ix_n^j\xi_n^j+it_n^j|\xi_n^j|^2}\psi^j(\frac{\cdot-x_n^j}{h_n^j})\right)(x-2t_n^j\xi_n^j).$$
Hence in the decomposition \eqref{eq:loc-6}, we see that $\xi_n^j$
no longer plays the role of the frequency parameter and hence we
can assume that $\xi_n^j\equiv 0$ for this $j$ term. With this
assumption,
$$D^{-s}g_n^j(\psi^j)=e^{-it_n^j\Delta}\left(\frac
{1}{(h_n^j)^{d/2-s}}(D^{-s}\psi^j)(\frac{\cdot-x_n^j}{h_n^j})\right)(x).$$
Setting $\phi^j:=D^{-s}\psi^j$, we see that this case is done.

{\bf Case 2.} $\lim_{n\to\infty}|h_n^j\xi_n^j|=\infty$. It is
clear that
\begin{equation*}
e^{it\Delta}g_n^j(\psi^j)(x)=e^{ix\xi_n^j-it|\xi_n^j|^2}\frac
{1}{(h_n^j)^{d/2}}e^{i\frac {t-t_n^j}{(h_n^j)^2}\Delta}
\psi^j(\frac {x+x_n^j-2t_n^j\xi_n^j}{h_n^j}).
\end{equation*}
Hence if changing $t'=\frac {t-t_n^j}{(h_n^j)^2}$ and $x'=\frac
{x+x_n^j-2t_n^j\xi_n^j}{h_n^j}$, we obtain
\begin{align*}
\|D^{-s}g_n^j(\psi^j)\|_{L^q_tL^r_x}&=(h_n^j)^{\frac 2q+\frac
dr-s-\frac d2} \|D^{-s}(e^{ixh_n^j\xi_n^j}
e^{is\Delta}\psi^j)\|_{L^q_{t'}L^r_{x'}}\\
&=\|D^{-s}(e^{ixh_n^j\xi_n^j}e^{is\Delta}\psi^j)\|_{L^q_{t'}L^r_{x'}}.
\end{align*}
By the H\"ormander-Mikhlin multiplier theorem \cite[Theorm 4.4]{Tao:2007-247A-fourier-analysis}, for $2\le r<\infty$,
\begin{equation*}
\|D^{-s} e^{it\Delta}g_n^j(\psi^j)\|_{L^q_tL^r_x}\sim \frac
{1}{(|h_n^j\xi_n^j|)^s} \|e^{is\Delta} \psi^j\|_{L^q_{t'}
L^r_{x'}}\lesssim_{\psi^j} (|h_n^j\xi_n^j|)^{-s}\to 0
\end{equation*} as $n$ goes to infinity since $s>0$ and $\psi^j$ is assumed to be Schwartz.
In view of this, we will organize $D^{-s}g_n^j(\psi^j)$ into the error term
$e^N_n$. Hence the decomposition \eqref{eq:ss-profile-decom} is
obtained. Finally the $\dot{H}^s_x$-orthogonality
\eqref{eq:ss-pf-L2-ortho} follows from \eqref{eq:pf-L2-ortho}, and
the constraint \eqref{eq:ss-param-constrnt} from Definition
\ref{def:orthogonal} since $\xi_n^j\equiv 0$ for all $j,n$.
Therefore the proof of Theorem \ref{thm:pf-homog-sobolev} is
complete.
\end{proof}

\appendix
\section{Proof of the improved Sobolev embeddings} Here we include the arguments for the generalized Keraani's and Killip-Visan's improved Sobolev embeddings, which can be used
to derive Theorem \ref{thm:pf-homog-sobolev} as well. Firstly the
generalization of \eqref{eq:keraani-impv-sobolev} is as follows:
for any $1<r<\infty$, $s\ge 0$ and $\frac 1r+\frac sd=\frac 12$,
\begin{equation}\label{eq:g-keraani-impv-sobolev}
\|f\|_{L^r_x}\lesssim
\|D^sf\|_{L^2_x}^{1-2s/d}\|D^sf\|_{\dot{B}^0_{2,\infty}}^{2s/d}.
\end{equation} To prove it, we will closely follow the approach in \cite{Bahouri-Gerard:1999:profile-wave} by Bahouri and G\'erard.
\begin{proof}[Proof of \eqref{eq:g-keraani-impv-sobolev}]
For every $A>0$, we define $f_{<A}$, $f_{>A}$ via
\begin{equation*}
\hat{f}_{<A}(\xi):=1_{|\xi|\le A}\hat{f}(\xi), \hat{f}_{>A}(\xi)=1_{|\xi|> A}\hat{f}(\xi).
\end{equation*}
From the Riemann-Lebesgue lemma,
\begin{equation*}\|f_{<A}\|_{L^{\infty}_x}\lesssim \|\hat{f}_{<A}\|_{L^1_\xi}
\lesssim \sum_{k\in \Z\atop k\le K} \|1_{2^k\le |\xi|\le
2^{k+1}}\hat{f}\|_{L^1_\xi},
\end{equation*} where $K$ is the largest integer such that $2^{K}\le A$.
Then by the Cauchy-Schwarz inequality,
\begin{equation*}
\|1_{2^k\le |\xi|\le 2^{k+1}}\hat{f}\|_{L^1_\xi}\lesssim 2^{k(\frac d2-s)}\|D^s f\|_{\dot{B}^{0}_{2, \infty}}.
\end{equation*}
Since $\frac d2-s>0$ and $2^K\le A$,
 $$\|f_{<A}\|_{L^{\infty}_x}\le C A^{\frac d2-s}\|D^s
f\|_{\dot{B}^{0}_{2,\infty}}$$
for some $C>0$. We write \begin{equation*}
\|f\|^r_{L^r_x}\sim\int_0^{\infty} \lambda^{r-1} |\{x\in \R^d:
|f|>\lambda\}| d\lambda
\end{equation*}
and set
\begin{equation*}
A(\lambda):=\left(\frac {\lambda}{2C\|D^s f\|_{\dot{B}^{0}_{2, \infty}}}\right)^{\frac {1}{d/2-s}}.
\end{equation*}
Thus we obtain $$\|f_{<A}\|_{L^{\infty}_x}\le \frac
{\lambda}{2}.$$ On the other hand, the Chebyshev inequality gives
$$|\{|f|>\lambda\}|\le |\{f_{>A(\lambda)}>\frac \lambda
2\}|\lesssim \dfrac
{\|\hat{f}_{>A(\lambda)}\|^2_{L^2_x}}{\lambda^2}.$$ Hence we have
\begin{align*}
\|f\|^r_{L^r_x}&\lesssim \int_{0}^{\infty} \lambda^{r-3} \int_{|\xi|>A(\lambda)} |\hat{f}(\xi)|^2d\xi d\lambda \\
& \lesssim \int_{0}^{\infty} |\hat{f}(\xi)|^2 \int_{0}^{2C\|D^s
f\|_{\dot{B}^{0}_{2, \infty}}|\xi|^{d/2-s}}\lambda^{r-3}
 d\lambda d\xi\\
& \lesssim \|D^s f\|^{r-2}_{\dot{B}^{0}_{2, \infty}}\int_{0}^{\infty} |\xi|^{2s} |\hat{f}(\xi)|^2 d\xi \lesssim \|D^s f\|^2_{L^2_x} \|D^s f\|^{r-2}_{\dot{B}^{0}_{2,
\infty}}.
\end{align*}
Therefore the proof of \eqref{eq:g-keraani-impv-sobolev} is
complete.
\end{proof}
Next we prove the generalized version of
\eqref{eq:killip-visan-impv-sobobev}: for any $1<r<\infty$, $s\ge
0$ and $\frac 1r+\frac sd=\frac 12$,
\begin{equation}\label{eq:g-killip-visan-impv-sobobev}
\|f\|_{L^r_x}\lesssim \|D^sf\|_{L^2_x}^{1-2s/d}\sup_{k\in \Z}\|f_k\|_{L^r_x}^{2s/d},
\end{equation}  where $f_k$ is defined as above. For the proof, we will closely follow the approach in \cite{Killip-Visan:2008:clay-lecture-notes}. We first recall the Littlewood-Paley square function estimate \cite[p.267]{Stein:1993}.
\begin{lemma}\label{le:lw-square}
Let $1<p<\infty$. Then for any Schwartz function $f$,
$$\|f\|_{L^p_x}\sim \|(\sum_{k\in \Z}|f_k|^2)^{1/2}\|_{L^p_x},$$
where $f_k$ is defined as in the introduction.
\end{lemma}
\begin{proof}[Proof of \eqref{eq:g-killip-visan-impv-sobobev}]
By Lemma \ref{le:lw-square}, we see that
\begin{equation}\label{eq:loc-5}
\|f\|^r_{L^r_x}\sim \int\left(\sum_{M\in \Z} |f_M|^2\right)^{r/2}dx.
\end{equation}
When $r\le 4$, we have $d-4s\ge 0$. Then the H\"older inequality and the Bernstein inequality
yield,
\begin{align*}
\|f\|^r_{L^r_x}
&\sim \int \left(\sum_{M\in \Z} |f_M|^2\right)^{r/4}\left(\sum_{N\in \Z} |f_N|^2\right)^{r/4}dx\\
&\lesssim \sum_{M\le N} \int |f_M|^{r/2}|f_N|^{r/2}dx\\
&\lesssim \sum_{M\le N} \|f_M\|_{L^{\frac {2d}{d-4s}}_x} \||f_M|^{r/2-1}\|_{L^{\frac ds}_x}
\||f_N|^{r/2-1}\|_{L^{\frac ds}_x}\|f_N\|_{L^2_x}\\
&\lesssim \left(\sup_{k\in \Z}\|f_k\|_{L^r_x}\right)^{r-2}\sum_{M\le N}\|f_M\|_{L^{\frac {2d}{d-4s}}_x}
\|f_N\|_{L^2_x}\\
&\lesssim \left(\sup_{k\in \Z}\|f_k\|_{L^r_x}\right)^{r-2} \sum_{M\le N} N^{-s}M^s\|D^sf_M\|_{L^2_x}\|D^sf_N\|_{L^2_x}.
\end{align*}
Then the Schur's test concludes the proof when $r\le 4$. On the
other hand, when $r>4$, we let $r^*=[r/2]$, the largest integer
which is less than $r/2$. Still by \eqref{eq:loc-5}, the H\"older
inequality and the Bernstein inequality, we have
\begin{align*}
\|f\|^r_{L^r_x}&\sim \int \left(\sum_{M_1}|f_{M_1}|^2\right)\left(\sum_{M_2}|f_{M_2}|^2\right)^{r^*-1}
\left(\sum_{M}|f_{M}|^2\right) ^{r/2-r^*} dx\\
&\lesssim \sum_{M_1\le M_2\le M_3\le\cdots\le M_{r^*-1}\le M} \int |f_{M_1}||f_{M_1}|f_{M_2}|^2|f_{M_3}|^2 \times\cdots\times|f_{M_{r^*-1}}|^2|f_{M}|^{r-2r^*}dx\\
&\lesssim \sum_{M_1\le M_2\le M_3\le\cdots\le M_{r^*-1}\le M}\|f_{M_1}\|_{L^r_x}\|f_{M_1}\|_{L^\infty_x}
\|f_M\|^{r-2r^*}_{L^r_x}\|f_{M_2}\|_{L^{r/2}_x}\times\\
&\times\|f_{M_2}\|_{L^r_x}\|f_{M_3}\|^2_{L^r_x}\cdots\|f_{M_{r^*-1}}\|^2_{L^r_x}\\
&\lesssim\left(\sup_{k\in \Z}\|f_k\|_{L^r_x}\right)^{r-2}\sum_{M_1\le M_2}M_1^{d/2-s}M_2^{s-d/2}
\|D^sf_{M_1}\|_{L^2_x}\|D^sf_{M_2}\|_{L^2_x}.
\end{align*}
Since $d/2-s>0$, the Schur's test again concludes the proof. Hence the proof of \eqref{eq:g-killip-visan-impv-sobobev} is complete.

\end{proof}

\bibliography{refs}
\bibliographystyle{plain}
\end{document}